\documentclass[a4paper,10pt,reqno]{amsart}

\usepackage[utf8]{inputenc}
\usepackage{amsthm}
\usepackage{amssymb}
\usepackage{amsmath}
\usepackage{mathtools}
\usepackage{pdfpages}
\usepackage{float}
\usepackage{tikz-cd}
\usepackage[all]{xy}
\usepackage{enumitem}
\usepackage{nicematrix}
\usepackage{fancyhdr}
\usepackage{caption}

\definecolor{green-url}{RGB}{0,128,0}
\definecolor{blue-url}{RGB}{0,0,205}
\definecolor{red-url}{RGB}{227,0,34}
\definecolor{ashgrey}{rgb}{0.7, 0.75, 0.71}
\usepackage[colorlinks=true,pdftex,unicode=true,linktocpage,bookmarksopen,hypertexnames=false,citecolor=green-url, linkcolor=blue-url, urlcolor=red-url]{hyperref}

\def\ZZ{ \mathbb{Z} }
\def\NN{ \mathbb{N} }

\newcommand{\cal}[1]{{\mathcal{#1}}}
\newcommand{\fk}[1]{{\mathfrak{#1}}}
\newcommand{\bff}[1]{{\mathbf{#1}}}

\newcommand{\defit}[1]{{\textsf{#1}}}

\newcommand{\vprod}{\sideset{}{\mkern -2mu\raisebox{0.2em}{$\vphantom{\prod}^v$}} \prod}

\newcommand{\Mod}[1]{\ (\mathrm{mod}\ #1)}

\makeatletter
\numberwithin{equation}{section}
\theoremstyle{plain}
\newtheorem{lemma}{Lemma}[section]
\newtheorem{theorem}[lemma]{Theorem}
\newtheorem{proposition}[lemma]{Proposition}
\newtheorem{corollary}[lemma]{Corollary}

\theoremstyle{definition}
\newtheorem{definition}[lemma]{Definition}
\newtheorem{example}[lemma]{Example}
\newtheorem{examples}[lemma]{Examples}

\theoremstyle{remark}
\newtheorem{remark}[lemma]{Remark}

\renewcommand{\epsilon}{\varepsilon}
\renewcommand{\theta}{\vartheta}
\renewcommand{\phi}{\varphi}

\setlist[enumerate,1]{itemsep=0.05cm, label=\textup{(}\arabic*\textup{)}}
\setlist[enumerate,2]{itemsep=0.05cm, label=\textup{(}\roman*\textup{)}}
\setlist[itemize,1]{itemsep=0.05cm}

\makeatletter
\@namedef{subjclassname@2020}{%
	\textup{2020} Mathematics Subject Classification}
\makeatother

\title[Every f.g.  abelian group is the class group of a generalized cluster algebra]{Every finitely generated abelian group is \\ the class group of a generalized cluster algebra}
\author{Mara Pompili}
\address{University of Graz, Department of Mathematics and Scientific Computing, NAWI Graz, Heinrichstrasse 36, 8010 Graz, Austria}
\email{mara.pompili@uni-graz.at}
\thanks{This work was partly supported by the Austrian Science Fund (FWF), project 10.55776/DOC-183-N.}

\begin{document}
	\begin{abstract}
		We determine the class group of those generalized cluster algebras that are Krull domains. In particular, this provides a criterion for determining whether or not a generalized cluster algebra is a UFD. In fact, any finitely generated abelian group can be realized as the class group of a generalized cluster algebra. Additionally, we show that generalized cluster algebras are FF-domains and that their cluster variables are strong atoms. Finally, we examine the factorization and ring-theoretic properties of Laurent phenomenon algebras.3
	\end{abstract}
	\maketitle
	\section{Introduction }
	The \defit{class group} of a \defit{Dedekind domain} (a domain where every non-zero ideal factors uniquely into prime ideals) is an invariant that measures the failure of unique factorization. In a unique factorization domain (UFD), every element can be uniquely factored into atoms (also called irreducibles), up to order and units. However, not all domains exhibit unique factorization, and class groups are a tool that captures this phenomenon.
	If a Dedekind domain allows non-unique factorizations of its elements into atoms, the class group provides a way of classifying the different types of factorizations.
	Even if the class group is typically defined for {Dedekind domains}, the definition can be extended to \defit{Krull domains,} a higher dimensional generalization of Dedekind domains. The class group of a Krull domain $A$ is the abelian group of the equivalence classes of divisorial ideals of $A$, where two ideals are equivalent if they differ by a principal ideal. The structure of a Krull domain can be fully described (up to units) by its class group and the distribution of prime divisors within each class. Claborn's Realization Theorem \cite[Theorem 7]{Cla66} states that every abelian group is isomorphic to the
	class group of a Dedekind domain. Similar realization theorems, yielding commutative Dedekind domains which are more geometric, respectively
	number theoretic, in nature, were obtained by Leedham-Green \cite{LG72} and Rosen \cite{Ros73,Ros76}. For more recent research on this direction see for instance \cite{CG24, Cha22, Cla09, Per23, Sme17}.
	One of the goals of this paper is to show that for any finitely generated abelian group $G$, there exists a \defit{generalized cluster algebra} that has $G$ as its class group. 
	
	Generalized cluster algebras were introduced by Chekhov and Shapiro \cite{CS13} in 2014. Their idea was to generalize the construction of cluster algebras introduced by Fomin and Zelevinsky in their four-part series of foundational papers \cite{FZ02,FZ03_2,BFZ05,FZ07} (the paper \cite{BFZ05} is in cooperation with Berenstein). Generalized cluster algebras were introduced to study Teichmüller spaces of Riemann surfaces with orbifold points.
	Unlike the classical setting, where exchange relations are binomial, generalized cluster algebras allow multinomial exchange relations. In \cite[Theorem 2.5 and Theorem 2.7]{CS13} the authors show that generalized cluster algebras also possess the Laurent phenomenon and their finite type classification coincides with the one of cluster algebras. 
	Gekhtman, Shapiro, and Vainshtein \cite[Theorem 4.1]{GSV18} prove that a generalized upper cluster algebra coincides with its upper
	bounds under coprimality conditions. In Section \ref{section:3} we provide a proof of this fact using a more elementary argument.
	Later, Bai, Chen, Ding, and Xu \cite[Theorem 3.10]{BCDX20} prove that
	the conditions of acyclicity and coprimality close the gap between lower bounds and
	upper bounds associated to generalized cluster algebras, and thus obtaining the
	standard monomial bases of these algebras.

	Despite these and other similarities between cluster algebras and generalized cluster algebras, our study reveals a significant difference in their factorization properties, particularly in terms of their class groups. 
The class groups of cluster algebras have been previously studied in \cite{GELS19,P24}. In \cite{GELS19}, the authors show that the class group of a cluster algebra that is a Krull domain is always of the form $\ZZ^r$. Furthermore, in the case of acyclic cluster algebras, they provide an explicit formula for 
$r$ based on the initial data. In \cite{P24}, similar results are obtained for upper cluster algebras that are Krull domains. Specifically, the authors present an explicit calculation of the rank of the class group of full rank upper cluster algebras. Generalized cluster algebras have a more complex behavior: their class groups are still finitely generated abelian groups (see Theorem \ref{thm:classgroupgeneralcase}), but they can have torsion (Example \ref{ex:Z2}). In fact, Theorem \ref{thm:realization} shows that every finitely generated abelian group can be realized as the class group of a generalized cluster algebra. In addition, every class contains infinitely many prime divisors.
	
	For a classical cluster algebra that is a Krull domain only two scenarios are possible: either the algebra is a UFD, or all arithmetic invariants are infinite and any non-empty finite subset $L \subseteq \NN_{\ge 2}$ can be
	realized as a set of lengths of some element. This dichotomy is not true anymore for a generalized cluster algebra that is a Krull domain, since its class group can be finite, and hence arithmetical invariants can be finite too. For instance, we provide an example of a generalized cluster algebra that is not a UFD, but it is half-factorial, that is there exists an element of the algebra that has at least two different factorizations but the lengths of the factorizations of each element are always the same.
	
	In addition to the study of the class groups of generalized cluster algebras, in Section \ref{section:2} we show that generalized cluster algebras are FF-domains (domains in which every element has only finitely many divisors) and that cluster variables are strong atoms (elements such that their powers factor uniquely into atoms), as was done for cluster algebras in \cite[Section 2]{P24}. Strong atoms have been the focus of numerous recent studies; see, for instance, \cite{Ang22, FFNSW24, GK24, RW21}.
	
	Moreover, Section \ref{section:LP} is dedicated to the study of some ring-theoretical properties of Laurent phenomenon algebras, a class of rings introduced by Lam and Pylyavskyy \cite{LP16}, another generalization of cluster algebras.

	\vspace{0.5cm}
	\subsection*{Notations} We denote by $\NN$ the set of positive integers, and by $\NN_0$ the set of non-negative integers. If $a,b\in \ZZ$, we let $[a,b]=\{x\in \ZZ\mid a\le x\le b\}$ be the (discrete) interval from $a$ to $b$. A \defit{domain} $A$ is a non-zero commutative ring with identity without non-zero zero divisors. We denote by $A^\bullet$ the set of non-zero elements of $A$, by $A^\times$ the set of units (or invertible elements) of $A$, and by $\mathbf{q}(A)$ its quotient field. We say that two elements $a,b\in A$ are associated if $a=\epsilon b$ for some $\epsilon \in A^\times$.

	\section{Generalized cluster algebras}\label{section:2}
	For some reference about cluster algebras and generalized cluster algebras, see \cite{BCDX20, Kel10, Wil14}. 
    
	Let $n,m\in \NN_0$ with $n+m>0$, and let $R$ be the ring of integers $\ZZ$ or a field of characteristic 0. We will refer to the rational function field $\cal{F}$ in $n+m$ independent variables over $\bff{q}(R)$ as the \defit{ambient field}.
	A \defit{cluster} is a set $\bff{x}=\{x_{1},\ldots,x_{n},x_{n+1},\ldots,x_{n+m}\}\in \mathcal{F}$ of algebraically independent variables over $R$, and an \defit{extended cluster} is a cluster with a bipartition of its element. We will denote by $L_\bff{x}$ the Laurent polynomial ring $R[x_{1}^{\pm1},\ldots,x_{n+m}^{\pm1}]$. 
	\begin{definition}[Exchange matrix]
		A matrix $B=(b_{ij})\in \mathcal{M}_{n\times n}(\ZZ)$ is \defit{skew-symmetrizable} if there exists a diagonal matrix $D\in \mathcal{M}_{n\times n}(\NN)$ such that $DB$ is skew-symmetric. 
		An $(n+m)\times n$ integer matrix is an \defit{exchange matrix} if the submatrix supported on the first $n$ rows is skew-symmetrizable. 
	\end{definition}
	
	Let $B$ be a $(n+m)\times n$ exchange matrix. We denote by $\bff{d}$ a $n$-tuple $(d_1,\ldots,d_n)$, where $d_i$ is a positive integer such that $b_{ji}/d_i\in \ZZ$ for every $i\in [1,n]$ and $j\in [1,n+m]$. We will write $\beta_{ji}:=b_{ji}/d_i.$
	
	\begin{definition}[Generalized seed]
		A \defit{generalized seed} of rank $n$ in $\cal{F}$ is a tuple $(\bff{x},\rho,B)$ where 
		\begin{itemize}
			\item  $\bff{x}=\{x_1,\ldots,x_n,\ldots,x_{n+m}\}$ is an {extended cluster}: the elements $\{x_1,\ldots,x_n\}$ are called \defit{cluster variables}, and the elements of $\{x_{n+1},\ldots,x_{n+m}\}$ are called \defit{frozen variables};
			\item $B=(b_{ij})\in \mathcal{M}_{(n+m)\times n}(\ZZ)$ is an exchange matrix;
			\item $\rho=\{\rho_i\mid i\in [1,n]\}$ is the set of \defit{strings}, where $\rho_i=\{\rho_{i,0},\ldots,\rho_{i,d_i}\}$ with $\rho_{i,0}=\rho_{i,d_i}=1$, and $\rho_{ij}$ is a monomial in $R[x_{n+1},\cdots,x_{n+m}]$ for every $j\in [1,d_i-1]$.
		\end{itemize}
	\end{definition}
	Let $(\bff{x},\rho,B)$ be a generalized seed. The \defit{directed graph} $\Gamma(\bff{x},\rho,B)$ is the graph whose vertices are all $i\in [1,n]$ and such that there is an edge \begin{tikzcd}[column sep=small]
		i \ar[r] & j
	\end{tikzcd} if and only if $b_{ij}>0$.

	For $x\in \mathbb{R}$, define the function $[x]_+:=x$ if $x\ge 0$, and $[x]_+=0$ otherwise. 
	
	\begin{definition}[Mutation of generalized seeds]
		For $i\in[1,n]$ the \defit{mutation} of a generalized seed $(\bff{x},\rho,B)$ in direction $i$ is a triple $\mu_i(\bff{x},\rho,B)=(\bff{x}_i,\rho',B_i)$ with
		\begin{itemize}
			\item $\bff{x}_i=(\bff{x}\setminus \{x_i\})\cup \{x_i'\}$ where $$x_i'=\frac{1}{x_i}\left(\sum_{j=0}^{d_i}\rho_{i,j}\prod_{k=1}^{n+m}x_k^{j[\beta_{ki}]_++(d_i-j)[-\beta_{ki}]_+}\right)\in \mathcal{F};$$
			\item $\rho'=(\rho\setminus \{\rho_i\})\cup \{\rho_i'\},$ where $\rho'_{ij}=\rho_{i,d_i-j}$ for $j\in[0,d_i]$;
			\item $B_i=(b'_{kl})$ with \[b_{kl}'=\begin{cases}-b_{kl}& \text{if $k=i$ or $l=i$;}\\ 	b_{kl}+([b_{il}]_+b_{ki} + b_{il}[-b_{ki}]_+) & \text{otherwise.} \end{cases} \]
		\end{itemize}The polynomials \[f_i:=\sum_{j=0}^{d_i}\rho_{i,j}\prod_{k=1}^{n+m}x_k^{j[\beta_{ki}]_++(d_i-j)[-\beta_{ki}]_+}\in R[\bff{x}]\] are called the \defit{exchange polynomials} (associated to the seed $(\bff{x},\rho,B)$).
	\end{definition}
	\begin{definition}\label{def:coprimegenseed}
		A generalized seed $(\bff{x},\rho,B)$ is 
		\begin{enumerate}
			\item 	\defit{coprime} if the exchange polynomials $f_{1},\ldots,f_{n}$ are pairwise coprime in $R[\bff{x}]$; 
			\item  \defit{acyclic} if the directed graph $\Gamma(\bff{x},\rho,B)$ does not contain directed cycles.
		\end{enumerate}
	\end{definition} Notice that, since $d_i$ divides $b_{ji}$ for every $j\in[1,n+m]$, each $d_i$ divides $b'_{ji}$ for every $j\in[1,n+m]$. Moreover, if $d_i=1$ for every $i\in[1,n]$ we get the classical mutation of a classical seed, as in \cite{FZ02}. 

    \begin{example}\label{ex:genmut}
        Let \[B=\begin{pmatrix}
            0 & 2 \\ -3 & 0 \\ 0 & -4
        \end{pmatrix}, \quad \bff{x}=\{x_1,x_2,x_3\},\quad (d_1,d_2)=(3,2),\quad \rho_1=\{1,3,3,1\},\quad \rho_2=\{1,2,1\},\] so that $(\bff{x},\rho,B)$ is a generalized seed. Notice that in this case $x_3$ is a frozen variable, thus we can mutate only in direction $1$ and $2$. Therefore, mutations only change the sign of the principal part of the matrix, i.e. the submatrix of $B$ supported on the first 2 rows. Recall that we denote by $\beta_{ji}$ the integers $b_{ji}/d_i$. The exchange relations are given by 
        \begin{equation*}
            \begin{split}
            f_1=x_1x_1'&=\prod_{k=1}^3x_k^{3[-\beta_{k1}]_+}+3\prod_{k=1}^3x_k^{[\beta_{k1}]_++2[-\beta_{k1}]_+}+3\prod_{k=1}^3x_k^{2[\beta_{k1}]_++[-\beta_{k1}]_+}+\prod_{k=1}^3x_k^{3[\beta_{k1}]_+}\\&=x_1^3+3x_1^2+3x_1+1,
            \end{split}
        \end{equation*} and
        \begin{equation*}
                f_2=x_2x_2'=\prod_{k=1}^3x_k^{2[-\beta_{k2}]_+}+2\prod_{k=1}^3x_k^{[\beta_{k2}]_++[-\beta_{k2}]_+}+\prod_{k=1}^3x_k^{3[\beta_{k2}]_+}=x_2^2+2x_2x_3^2+x_3^4.
        \end{equation*} 
    \end{example}
    \begin{examples}
    \begin{enumerate}
        \item The seed in Example \ref{ex:genmut} is coprime ($f_1,$ and $f_2$ do not share any non-trivial factor), and acyclic, since $\Gamma(\bff{x},\rho,B)= 2 \overset{2}{\rightarrow} 1$ and it does not contain directed cycles.
        \item Let \[B=\begin{pmatrix}
            0 & 1 & 0 \\ -1 & 0 & 1 \\ 0 & -1 & 0
        \end{pmatrix}, \quad \bff{x}=\{x_1,x_2,x_3\}\] be a seed. Then \[f_1=x_2^2+x_3^2,\quad f_2=x_1^2+x_3^2,\quad f_3=x_2^2+x_3^2.\] Then the seed $(\bff{x},B)$ is not coprime ($f_1=f_3$), but it is acyclic ($\Gamma(\bff{x},B)=1\rightarrow 2\rightarrow 3$).
        \item Let \[B=\begin{pmatrix}
            0 & 2 & -2 \\ -2 & 0 & 2 \\ 2 & -2 & 0
        \end{pmatrix}, \quad \bff{x}=\{x_1,x_2,x_3\}\] be a seed. Then \[f_1=x_2+1,\quad f_2=x_1+x_3,\quad f_3=x_2+1.\] Then the seed $(\bff{x},B)$ is coprime if $R=\ZZ$, but it is not acyclic as $$\Gamma(\bff{x},B)=\begin{tikzcd}[row sep=50pt]
            & 1 \ar[dr,bend left=15] \ar[dr, bend right=15] & \\
            3 \ar[ur, bend left=15] \ar[ur, bend right=15] & & 2. \ar[ll, bend left=15] \ar[ll, bend right=15]
        \end{tikzcd}$$
        \end{enumerate}
    \end{examples}
    
    One can prove that $\mu_i(\bff{x},\rho,B)$ is a generalized seed with the same ambient field as $(\bff{x},\rho,B)$ and that $(\mu_i\circ \mu_i)(\bff{x},\rho,B)=(\bff{x},\rho,B).$ 
	Mutations induce an equivalence relation on seeds. We say that two generalized seeds $(\bff{x},\rho,B)$ and $(\mathbf{y},\nu,C)$ are \defit{mutation-equivalent} if there exist $i_1,\dots,i_k\in [1,n]$ such that $(\bff{y},\nu,C)=(\mu_{i_1}\circ\cdots\circ \mu_{i_k})(\bff{x},\rho,B)$.
	In this case, we write $(\bff{x},\rho,B)\sim (\mathbf{y},\nu,C),$ or, if no confusion can arise, $\bff{x}\sim\bff{y}$. Denote by $\mathcal{M}(\mathbf{x},\rho,B)$ the mutation equivalence class of $(\mathbf{x},\rho,B)$ and by $\mathcal{X}=\mathcal{X}(\mathbf{x},\rho,B)$ the set of all cluster variables appearing in $\mathcal{M}(\mathbf{x},\rho,B).$
	
	\begin{definition}
		Let $(\bff{x},\rho,B)$ be a generalized seed. 
        \begin{enumerate}
            \item The \defit{generalized cluster algebra} associated to $(\mathbf{x},\rho,B)$ is the $R$-algebra 
		$$\cal{A}(\mathbf{x},\rho,B)=R[x_{n+1}^{\pm1},\ldots,x_{n+m}^{\pm1}][x\mid x\in \mathcal{X}].$$ 
            \item The \defit{generalized upper cluster algebra} associated to $(\mathbf{x},\rho,B)$ 
		is the intersection \[\cal{U}(\bff{x},\rho,B)=\bigcap_{\bff{x}\sim \bff{y}}L_{\bff{y}}\] 
		of the Laurent polynomials rings associated to each cluster.
        \end{enumerate}
		The elements $x\in \mathcal{X}$ are the \defit{cluster variables} of $\cal{A}(\mathbf{x},\rho,B)$ and
		the elements $x_{n+1}^{\pm1},\ldots,x_{n+m}^{\pm1}$ are the \defit{frozen variables.}
	\end{definition}

	\begin{theorem}[{\cite[Theorems 2.5 and 2.7]{CS13}}] Let $(\bff{x},\rho,B)$ be a generalized seed.
    \begin{enumerate}[label=\textup(\normalfont{\arabic*}\textup)]
        \item Every cluster variable can be expressed as a Laurent polynomial in every cluster, or equivalently $\cal{A}(\bff{x},\rho,B)\subseteq \cal{U}(\bff{x},\rho,B).$
        \item The generalized cluster algebra $\cal{A}(\bff{x},\rho, B)$ has only finitely many cluster variables if and only if the directed graph $\Gamma(\bff{x},\rho, B)$ is a Dynkin diagram of finite type.
    \end{enumerate}
	\end{theorem}
    \begin{example}[{\cite[Theorem 2.7]{CS13}}]
        Let $B=\begin{pmatrix}
            0 & 1 \\ -2 & 0
        \end{pmatrix}.$ Let $(d_1,d_2)=(2,1),$ and $\rho_1=\{1,a,1\},$ $a\in \ZZ$, and $\rho_2=\{1,1\}.$ Let $\bff{x}=\{x_1,x_2\}$ be a cluster. Let $x_k$ with ${k\in \ZZ}$ be the cluster variable obtained mutating $x_{k-2}$ in direction $1$, resp. $2$, if $k$ is odd, resp. if $k$ is even. Then $\{x_k\}_{k\in \ZZ}$ satisfy the exchange relations  \begin{equation*}
            x_{k-1}x_{k+1}=\begin{cases}
                1+x_k & \text{if } k\in 2\ZZ\\ 1+hx_k+x_k^2 & \text{otherwise}.
            \end{cases}
        \end{equation*}
        By a direct calculation, we get that the sequence $\{x_k\}_{k\in \ZZ}$ is $6$-periodic. Therefore \[\cal{A}(\bff{x},\rho,B)=R[x_1,x_2,x_3,x_4,x_5,x_6].\] Moreover, notice that the seed is coprime and acylic.
    \end{example}

	\begin{definition} Let $(\bff{x},\rho,B)$ be a generalized seed. For every $i\in [1,n]$ denote by $\mathbf{x}_i$ the seed $\{x_1,\ldots,x_i',\ldots,x_{n+m}\}$ obtained mutating $\mathbf{x}$ in direction $i$.
		\begin{enumerate}[label=(\roman*)]
			\item   The \defit{upper bound} associated to $(\bff{x},\rho,B)$ is the domain $S_{\bff{x}}=\bigcap_{i=0}^{n}L_{\bff{x}_{i}}$.
			\item The \defit{starfish product} associated to $(\bff{x},\rho,B)$ is the domain $P_t=\prod_{i=0}^nL_{\bff{x}_i}$.
		\end{enumerate}
	\end{definition}
	Note that $\cal{A}\subseteq \cal{U} \subseteq S_{\bff{x}}$ and $\cal{A}$, $\cal{U}\hookrightarrow P_\bff{x}$ for every cluster $\bff{x}$.

	\subsection{Factorization properties of generalized cluster algebras} 
	Let us recall some definitions. For more details about factorization theory, see for instance  \cite{GH06}.
	\begin{definition}Let $A$ be a domain.
		\begin{enumerate}
			\item A non-unit $u\in A^\bullet$ is an \defit{atom}, or an \defit{irreducible element}, if $u=ab$ with $a,b\in A$ implies $a\in A^\times$ or $b\in A^\times$.
			\item An atom $u\in A^\bullet$ is a \defit{strong atom}, or an \defit{absolutely irreducible element}, if for all positive integers $n\ge 1$ the only factorization (up to associates) of $u$ is  $u^n=u\cdots u.$  
			\item A non-unit $p\in A^\bullet$ is a \defit{prime element} if $p$ dividing $ab$ with $a,b\in A$ implies that $p$ divides one of $a$ and $b$.
		\end{enumerate}
	\end{definition}
	A domain $A$ is \defit{atomic} if every non-zero non-unit element of $A$ can be written as a finite product of atoms. It is an \defit{FF-domain} (finite factorization domain) if it is atomic and every non-zero non-unit factors into atoms in only finitely many ways up to order and associates. Finally, $A$ is a \defit{factorial} domain (or a \defit{UFD}) if it is atomic and every non-unit of $A$ factors in a unique way up to order and associates.
	
	As for cluster algebras, the following properties hold. 
	\begin{proposition}\label{prop:unit}
		Let $(\bff{x},\rho,B)$ be a generalized seed, and $\cal{A}$, $\cal{U}$ be the generalized cluster algebra, and the generalized upper cluster algebra associated to $(\bff{x},\rho,B), $ respectively. Then:
		\begin{enumerate}[label=\normalfont{(\arabic*)}]
			\item The units of $\cal{A}$ and $\cal{U}$ are the set $\{\epsilon x_{n+1}^{a_{n+1}}\cdots x_{n+m}^{n+m}\mid \epsilon\in R^{\times}, a_i\in \ZZ\}.$
			\item Every cluster variable is a strong atom of $\cal{A}$ and $\cal{U}.$ 
			\item Two cluster variables $x$ and $y$ are associate if and only if $x=y$.
			\item The following statements are equivalent 
			\begin{enumerate}[label=\textup(\normalfont{\alph*}\textup)]
				\item $\cal{A}$ \textup(resp. $\cal{U}$\textup) is factorial;
				\item every cluster variable is prime in $\cal{A}$ \textup(resp., $\cal{U}$\textup);
				\item the cluster variables $x_1,\ldots,x_n$ are prime in $\cal{A}$ \textup(resp., $\cal{U}$\textup).
			\end{enumerate}
			\item Assume that $\cal{A}$ \textup(resp., $\cal{U}$\textup) is factorial, then
			\begin{enumerate}[label=\normalfont{(\roman*)}]
				\item the exchange polynomials associated to a cluster are pairwise distinct;
				\item all exchange polynomials are irreducible. 
			\end{enumerate}
		\end{enumerate}
	\end{proposition}
	\begin{proof}
		For (1), (3), (5) repeat the arguments in \cite{GLS13}, for (2), repeat the argument in \cite[Corollary 2.5]{P24}, and for (4) repeat the argument in \cite[Corollary 1.23]{GELS19}
	\end{proof}

	\begin{proposition}\label{prop:FF}
		Let $(\bff{x},\rho, B)$ be a generalized seed, and $A$ be either the generalized cluster algebra $\cal{A}(\bff{x},\rho, B)$ or the generalized upper cluster algebra $\cal{U}(\bff{x},\rho, B)$. Then $A$ is an FF-domain. In particular, it is atomic.
	\end{proposition}
	\begin{proof}
		To prove the statement it suffices to show that every non-zero element of $A$ has finitely many non-associated divisors \cite[Proposition 1.5.5]{GH06}. Suppose then for sake of contradiction that there is a non-zero element $a\in A$ with infinitely many non-associated divisors. Let $\bff{x}$ be a cluster and $P_\bff{x}$ the starfish product associated to it. Since every divisor of $a$ in $A$ is a divisor of $a$ in $L_{\bff{x}_i}$ for every $i\in [0,n]$, we have that, if $b$ divides $a$ in $A$, then $(b,\ldots, b)$ divides $(a,\ldots, a)$ in $P_\bff{x}$. The ring $P_\bff{x}$ is factorial, so in particular the element $(a,\ldots, a)$ has only finitely many non-associated divisors in $P_\bff{x}$, hence there must exist a sequence $(b_k)_{k\ge 1}$ consisting of non-associated divisors of $a$ in $A$ such that \[(b_k\ldots,b_k)P_\bff{x}=(b_l,\ldots,b_l)P_\bff{x}\quad \text{for every}\, k,l\ge 1.\] This implies that, for every $k,l\ge 1$, the elements $b_k$ and $b_l$ are associated in ${L_{\bff{x}_i}}$ for every $i\in [0,n]$, in particular, for every pair $(j,i)\in [1,n+m]\times [0,n]$ there exist $a_{j,i}\in \ZZ$ and $\epsilon_i\in R^\times$ such that \[\epsilon_ix_1^{a_{1,i}}\cdots {x_i'}^{a_{i,i}}\cdots x_{n+m}^{a_{{n+m},i}}b_k=b_l.\] Then for every $i\in [1,n]$ we have \[\epsilon_ix_1^{a_{1,i}}\cdots {x_i'}^{a_{i,i}}\cdots x_{n+m}^{a_{{n+m},i}}b_k=\epsilon_0x_1^{a_{1,0}}\cdots {x_i}^{a_{i,0}}\cdots x_{n+m}^{a_{{n+m},0}}b_k,\] whence $a_{j,i}=a_{j,0}$ and $a_{i,i}=a_{i,0}=0$, and $\epsilon_i=\epsilon_0$ for every $(j,i)\in[1,n+m]\times [1,n]$. This implies that $b_k=\epsilon_0x_{n+1}^{a_{n+1,0}}\cdots x_{n+m}^{a_{n+m,0}} b_l,$ and this is a contradiction because we assumed that $b_k$ and $b_l$ are not associated in $P_\bff{x}$.
	\end{proof}

	\section{Starfish Lemma}
	This section aims to prove the Starfish Lemma for generalized cluster algebras. This has been already done in \cite{GSV18} for a more general setting, but here we present a more straightforward proof that follows the original proof of the Starfish Lemma for classical cluster algebras, as in \cite{BFZ05}.
	
	\begin{proposition}\label{prop:star5}\cite[Lemma 3.9]{BCDX20}
		Let $(\bff{x},\rho,B)$ be a coprime generalized seed. Then \begin{equation}\label{eq:star5}
		S_\bff{x}=\bigcap_{i=2}^n R[x_{1},x_1',x_2^{\pm1},\cdots,x_{i-1}^{\pm1},x_i,x_i',x_{i+1}^{\pm1},\cdots,x_{n+m}^{\pm1}]
		\end{equation}
	\end{proposition}
	
	\begin{lemma}\label{lemma:4}
		Let $(\bff{x},\rho,B)$ be a generalized seed. Let $x_2'$ and $x_2''$ be the cluster variables exchanged with $x_{2}$ in the generalized seeds $\mu_2(\bff{x},\rho,B)$ and $\mu_2\mu_1(\bff{x},\rho,B)$, respectively, then \begin{equation}\label{eq:lemma4}
		R[x_1,x_1',x_2,x_2',x_3^{\pm1},\ldots,x_{n+m}^{\pm1}]=R[x_1,x_1',x_2,x_2'',x_3^{\pm1},\ldots,x_{n+m}^{\pm1}]
		\end{equation}
	\end{lemma}Before proceeding with the proof, we present an example to help the reader navigate the technical details of the lemma.
    \begin{example}
        Let \[B=\begin{pmatrix}
            0 & 2 \\ -3 & 0 \\ 0 & -4
        \end{pmatrix}, \quad \bff{x}=\{x_1,x_2,x_3\},\quad \rho_1=\{1,3,3,1\},\quad \rho_2=\{1,2,1\},\] as in Example \ref{ex:genmut}. Let $x_2'$ and $x_2''$ be the cluster variables exchanged with $x_{2}$ in the generalized seeds $\mu_2(\bff{x},\rho,B)$ and $\mu_2\mu_1(\bff{x},\rho,B)$, respectively. Observe that \[\mu_1(B)=\begin{pmatrix}
            0 & -2 \\ 3 & 0 \\ 0 & -4
        \end{pmatrix}.\] Therefore we have  \begin{gather*}
		x_1x_1'=x_2^3+3x_2^2+3x_2+1\\
		x_2x_2'=x_1^2+2x_1x_3^2+{x_3^4},\\
		x_2x_2''= {x_1'}^2x_3^4+2{x_1'}x_3^2+1.
		\end{gather*} Then if we replace $x_3^4$ by $x_2x_2'-x_1^2+2x_1x_3^2$ we get \begin{equation*}
		    \begin{split}
		        x_2x_2''&={x_1'}^2x_3^4+2{x_1'}x_3^2+1\\& =(x_2x_2'-x_1^2-2x_1x_3^2){x_1'}^2+2{x_1'}x_3^2+1\\&=x_2x_2'{x_1'}^2-(x_2^3+3x_2^2+3x_2+1)^2-2x_3^2x_1'(x_2^3+3x_2^2+3x_2+1)+2{x_1'}x_3^2+1\\&=x_2x_2'{x_1'}^2-(x_2^3+3x_2^2+3x_2)^2-2x_2^3-6x_2^2-6x_2+2x_3^2x_1'(x_2^3+3x_2^2+3x_2)
		    \end{split}
		\end{equation*} and in particular \[x_2''=x_2'{x_1'}^2-(x_2^2+3x_2+3)^2-2x_2^2-6x_2-6+2x_3^2x_1'(x_2^2+3x_2+3)\in R[x_1,x_1',x_2,x_2'].\]
        Similarly, \begin{equation*}
		    \begin{split}
		        x_2x_2'&=x_1^2+2x_1x_3^2+{x_3^4}\\& =(x_2x_2''-{x_1'}^2x_3^4-2{x_1'}x_3^2)x_1^2+2x_1x_3^2+{x_3^4}\\&=x_2x_2''x_1^2-x_3^4(x_2^3+3x_2^2+3x_2+1)^2-2x_3^2{x_1}(x_2^3+3x_2^2+3x_2+1)+2x_1x_3^2+{x_3^4}\\&=x_2x_2''x_1^2-x_3^4(x_2^3+3x_2^2+3x_2)^2-2x_3^4(x_2^3+3x_2^2+3x_2)+x_3^2x_1(x_2^3+3x_2^2+3x_2)
		    \end{split}
		\end{equation*} and therefore \[x_2'=x_2''x_1^2-x_3^4(x_2^2+3x_2+3)^2-2x_3^4(x_2^2+3x_2+3)+x_3^2x_1(x_2^2+3x_2+3)\in R[x_1,x_1',x_2,x_2''].\]
    \end{example}
	\begin{proof}[Proof of Lemma \ref{lemma:4}]
		We can freeze the variables $x_3,\ldots,x_{n+m}$ and view $R[x_3^{\pm1},\ldots,x_{n+m}^{\pm1}]$ as the new ground ring $R$. Thus we reduce the proof of \eqref{eq:lemma4} to the case $n=2$, that is, we will prove that \begin{equation}\label{eq:lemma4.1}
		R[x_1,x_1',x_2,x_2']=R[x_1,x_1',x_2,x_2''].
		\end{equation}
		We first show that $x_2''\in R[x_1,x_1',x_2,x_2'].$
		
		If $b_{12}=b_{21}=0,$ then $x_2''=x_2'$, which implies $x_2''\in R[x_1,x_1',x_2,x_2']$.
		
		Without loss of generality, assume that  $b_{12}> 0$. Write $b=b_{12}$ and $c=-b_{21}$. Denote by $B'=(b_{ij}')$ the matrix obtained by mutating $B$ in direction 1. Observe that $b_{21}'=c$ and $\mu_2(b'_{12})=b$. As always we denote by $\beta_{ji}$ the integer $b_{ji}/d_i$, with $d_1,\ldots,d_n\in \NN$ the given divisors of the columns of $B$. The exchange relations take then the form: 
		\begin{gather*}
		x_1x_1'=x_{2}^{c}\underbrace{\prod_{b_{j1}<0}x_{j}^{-b_{j1}}}_{q_1}+\underbrace{\prod_{b_{j1>0}}x_{j}^{b_{j1}}}_{r_1}+\sum_{k=1}^{d_{1}-1}\rho_{1,k}x_{2}^{(k-d_{1})\beta_{21}}\prod_{j=3}^{2+m}x_{j}^{k[\beta_{j1}]_{+}+(d_{1}-k)[-\beta_{j1}]_{+}},\\
		x_2x_2'=x_{1}^{b}\underbrace{\prod_{b_{j2}>0}x_{j}^{b_{j2}}}_{q_2}+\underbrace{\prod_{b_{j2<0}}x_{j}^{-b_{j2}}}_{r_2}+\sum_{k=1}^{d_{2}-1}\rho_{2,k}x_{1}^{k\beta_{12}}\prod_{j=3}^{2+m}x_{j}^{k[\beta_{j2}]_{+}+(d_{2}-k)[-\beta_{j2}]_{+}},\\
		x_2x_2''={x'_{1}}^{b}\underbrace{\prod_{b'_{j2}<0}x_{j}^{-b'_{j2}}}_{q_3}+\underbrace{\prod_{b'_{j2>0}}x_{j}^{b'_{j2}}}_{r_3}+\sum_{k=1}^{d_{2}-1}\rho_{2,k}{x'_{1}}^{(k-d_{2})\beta_{12}}\prod_{j=3}^{2+m}x_{j}^{k[\beta'_{j2}]_{+}+(d_{2}-k)[-\beta'_{j2}]_{+}},
		\end{gather*}
		We can rewrite the exchange relations as 		
		\begin{gather*}
		x_1x_1'=f(x_2)+r_1,\\
		x_2x_2'=g(x_1)+r_2,\\
		x_2x_2''=h(x_1')+r_3,
		\end{gather*}
		where 
		\begin{gather*}
		f(x_2)=\sum_{k=0}^{d_1-1}f_kx_{2}^{(k-d_{1})\beta_{21}},\qquad f_k=\rho_{1,k}\prod_{j=3}^{2+m}x_{j}^{k[\beta_{j1}]_{+}+(d_{1}-k)[-\beta_{j1}]_{+}}, \\
		g(x_1)=\sum_{k=1}^{d_{2}}g_kx_{1}^{k\beta_{12}},\qquad g_k=\rho_{2,k}\prod_{j=3}^{2+m}x_{j}^{k[\beta_{j2}]_{+}+(d_{2}-k)[-\beta_{j2}]_{+}}, \\
		h(x_1')=\sum_{k=0}^{d_{2}-1}h_k{x'_{1}}^{(d_{2}-k)\beta_{12}}, \qquad h_k=\rho_{2,k}\prod_{j=3}^{2+m}x_{j}^{k[\beta'_{j2}]_{+}+(d_{2}-k)[-\beta'_{j2}]_{+}},
		\end{gather*} 
		are polynomials without constant terms.
		
		Remember that, by definition of mutation, we have \[b'_{j2}=b_{j2}+([b_{12}]_+ b_{j1} + b_{12}[-b_{j1}]_+)=b_{j2}+b_{12}b_{j1} + b_{12}[-b_{j1}]_+.\] It is easy to check that \begin{equation*}\label{eq:r3}
		r_{2}r_{3}=q_2q_3r_{1}^{b}.
		\end{equation*} Moreover, every $h_{k}$ with $k\in [1,d_{2}-1]$ satisfies \begin{equation*}
		h_{k}=\frac{g_{k}q_3r_{1}^{r\beta_{12}}}{r_{2}}.
		\end{equation*} 
		Therefore we have the following equalities:
		\begin{equation}\label{eq:13}
		\begin{split}
		x_2x_2''&=q_3{x'_{1}}^{b}+\sum_{k=1}^{d_{2}-1}h_k{x'_{1}}^{(d_{2}-k)\beta_{12}}+r_3\\&=\frac{q_3}{r_2}\left(x_2x_2'-g\left(x_1\right)\right){x'_{1}}^{b}+\sum_{k=1}^{d_{2}-1}h_k{{x'_{1}}}^{(d_{2}-k)\beta_{12}}+r_3\\&=x_2\left(\frac{q_3x_2'{x'_1}^b}{r_2}\right)-\left(\frac{q_3g\left(x_1\right){x'_1}^b}{r_2}-\left(\sum_{k=1}^{d_2-1}h_k{x'_1}^{(d_2-k)\beta_{12}}\right)-r_3\right).
		\end{split}
		\end{equation}
		Observe that \begin{equation}\label{eq:10}
		\begin{split}
		\frac{q_3g\left(x_1\right){x'_1}^b}{r_2}&=\left(\sum_{k=1}^{d_2}\frac{g_kq_3}{r_2}x_1^{k\beta_{12}}\right){x'_1}^b=\sum_{k=1}^{d_2}\frac{g_kq_3}{r_2}(x_1x_1')^{k\beta_{12}}{x'_1}^{(d_2-k)\beta_{12}}\\&=\sum_{k=1}^{d_2}\frac{g_kq_3}{r_2}(f(x_2)+r_1)^{k\beta_{12}}{x'_1}^{(d_2-k)\beta_{12}}.
		\end{split}
		\end{equation} 
		The element $\frac{q_2q_3}{r_2}(f(x_2)+r_1)^{b}$ can we rewritten as
		\begin{equation}\label{eq:11}
		x_2P_{d_2}+\frac{q_2q_3r_1^b}{r_2}=x_2P_{d_2}+r_3,
		\end{equation}
		where $P_{d_2}$ is a polynomial in $x_2$. In addition, for every $i\in[1,d_2-1]$ we have 
		\begin{equation}\label{eq:12}
		\begin{split}
		\frac{g_{k}q_3}{r_2}(f(x_2)+r_1)^{k\beta_{12}}{x'_1}^{(d_2-k)\beta_{12}}=x_2P_{k}+\frac{g_{k}q_3r_1^{k\beta_{12}}}{r_2}{x'_1}^{(d_2-k)\beta_{12}}=x_2P_{k}+h_k{x'_1}^{(d_2-k)\beta_{12}},
		\end{split}
		\end{equation}
		where $P_r$ is a polynomial in $x_2$. Combining \eqref{eq:10} with \eqref{eq:11} and \eqref{eq:12}, we get
		\begin{equation*}
		\begin{split}
		\frac{q_3g\left(x_1\right){x'_1}^b}{r_2}&=\sum_{k=1}^{d_2}\frac{g_kq_3}{r_2}(f(x_2)+r_1)^{k\beta_{12}}{x'_1}^{(d_2-k)\beta_{12}}\\&=x_2P_{d_2}+r_3+\sum_{k=1}^{d_2-1}(x_2P_{k}+h_k{x'_1}^{(d_2-k)\beta_{12}}).
		\end{split}
		\end{equation*}
		Thus, \eqref{eq:13} can be rewritten as 
		\begin{equation*}
		x_2x_2''=x_2\left(\frac{q_3x_2'{x'_1}^b}{r_2}\right)+x_2\sum_{k=1}^{d_2}P_k. 
		\end{equation*}
		Hence $x_2''\in R[x_1,x_1',x_2,x_2']$. 
		
		Similarly, one can prove that $x_2'\in R[x_1,x_1',x_2,x_2'']$. 			
	\end{proof}
	
	\begin{theorem}\label{thm:star7}
		Let $(\bff{x},\rho, B)$ be a generalized seed, and $(\bff{x}',\rho', B')$ be the seed obtained mutating $(\bff{x},\rho, B)$ in direction $k\in [1,n]$. If the two seeds are both coprime, then $S_\bff{x}=S_{\bff{x}'}$. 
	\end{theorem}
	\begin{proof}
		The theorem is a direct consequence of Proposition \ref{prop:star5} and Lemma \ref{lemma:4}.
	\end{proof}
	\begin{corollary}
		Assume that all the generalized seeds mutation-equivalent to a generalized seed $(\bff{x},\rho, B)$ are coprime. Then the upper bound $S_\bff{x}$ is independent of the choice of generalized seeds mutation-equivalent to $(\bff{x},\rho, B)$, and so is equal to the generalized upper cluster algebra $\cal{U}(\bff{x},\rho, B).$
	\end{corollary}
	
	A classical seed $(\bff{x},B)$ is coprime if and only if no two
	columns of $B$ are proportional to each other with the proportionality coefficient being
	a ratio of two odd integers \cite[Lemma 3.1]{BFZ05}.
	For generalized seeds, this is not true anymore as the following example shows.
	
	\begin{example}
		Let $(\bff{x},\rho,B)$ be the generalized seed given by
		\begin{itemize}
			\item $\bff{x}=\{x_1,x_2,x_3\}$,
			\item $B=\begin{pmatrix}
			0 & -1 & 0 \\ 1 & 0 & -3 \\ 0 & 3 & 0
			\end{pmatrix},$
			\item $\rho_1=\rho_2=\{1,1\}$ and $\rho_3=\{1,1,4,1\}.$
		\end{itemize}
		Observe that the first and the last column of $B$ are proportional to each other with $3$ as a proportionality coefficient. On the other hand, the exchange polynomials are given by 
		\[f_1=1+x_2, \quad f_2=x_1+x_3^3,\quad f_3=x_2^3+x_2^2+4x_2+1,\] hence they are coprime.
	\end{example}
	
	\begin{lemma}\label{lemma:polytope}
		Let $(\bff{x},\rho,B)$ be a generalized seed. If no two columns of $B$  are proportional to each other, then the seed is coprime.
	\end{lemma}
	\begin{proof}
		For every $f=\sum_{k}c_{k}\bff{x}^{\bff{a}_{k}}\in R[\bff{x}]$, we denote by $N(f)\subseteq \mathbb{R}^{n+m}$ the Newton polytope of $f$, that is \[N(f)=\left\{\sum_{k}\alpha_{k}\bff{a}_{k}\mid \sum_{k}\alpha_{k}=1, \alpha_{k}\in \mathbb{R}_{\ge 0}\right\}.\]  Note that the Newton polytope $N(fg)$ of a product is equal to the Minkowski sum $N(f) + N(g)$. Denote by $B(j)$ the $j$-th column of $B$.
		
		Let $i\in[1,n]$ and $f_{i}$ be an exchange polynomial. By definition, we can write $f_{i}=\sum_{r=0}^{d_{i}}\rho_{r,i}\bff{x}^{\bff{a}_{r,i}}.$ Observe that, if $r\in [1,d_{i}-1]$, then \[\bff{a}_{r,i}=\frac{r}{d_{i}}\bff{a}_{0,i}+\frac{d_{i}-r}{d_{i}}\bff{a}_{d_{i},i},\] that is, $\bff{a}_{r,i}$ belongs to the line segment with endpoints $\bff{a}_{0,i}$ and $\bff{a}_{d_{i},i}$, and hence $N(f_{i})$ is a line segment. Furthermore, $N(f_{i})$ is parallel to $B(i)$. It follows that, for every non-trivial factor $g$ of $f_{i}$, the Newton polytope $N(g)$ is also a line segment parallel to $N(f_{i})$. Suppose now that, by contradiction, there exist $i\ne j$ such that $f_{i}$ and $f_{j}$ have a non-trivial common factor. Then $N(f_{i})$ and $N(f_{j})$ are collinear and so $pB(i)=qB(j)$, for some $p,q\in \mathbb{Q}.$ Hence a contradiction.\end{proof}
	
	Combining Lemma \ref{lemma:polytope} with the fact that matrix mutations preserve the rank of the matrix (\cite[Lemma 3.2]{BFZ05}), one can prove the following proposition.
	\begin{proposition}\label{prop:fullrank}
		Let $(\bff{x},\rho,B)$ be a generalized seed. If $B$ has full rank, then all seeds are coprime. In particular, we have that $\mathcal{U}=S_{\bff{y}}$ for every cluster $\bff{y}$.
	\end{proposition}
	We have shown that if a seed $(\mathbf{x},\rho,B)$ is coprime, then the generalized upper cluster algebra $\mathcal{U}(\mathbf{x},\rho,B)$ coincides with its upper bound $S_\bff{x}$. In \cite{BCDX20} the authors proved that if $(\mathbf{x},\rho,B)$ is acyclic and coprime, then the upper bound $S_\bff{x}$ coincide with its lower bound, that is the $R$-algebra generated by the elements $\{x_1,x_1',\ldots,x_n,x_n',x_{n+1},\ldots,x_{n+m}\}.$ More precisely,
	\begin{theorem}[{\cite[Theorem 3.10]{BCDX20}}]\label{thm:acycliccoprime}
		Let $(\bff{x},\rho,B)$ be a generalized seed. If the seed is acyclic and coprime, then \[R[x_1,x_1',\ldots,x_n,x_n',x_{n+1}^{\pm1},\ldots,x_{n+m}^{\pm1}]=\cal{A}(\bff{x},\rho,B)=\cal{U}(\bff{x},\rho,B)=S_\bff{x}.\]
	\end{theorem}
	
	\section{Generalized cluster algebras and their class group}\label{section:3}
	\subsection{The class group of a generalized cluster algebra}
	Let $A$ be a domain and $\mathfrak{X}(A)$ be the set of its height-1 prime ideals. The domain $A$ is a \defit{Krull domain} if the localization $A_\mathfrak{p}$ at every height-1 prime ideal $\fk{p}$ is a discrete valuation domain, $A$ is the intersection of all these DVRs and every non-zero element $a\in A$ is contained in at most a finite number of height-1 prime ideals of $A$.
	The \defit{class group} $\cal{C}(A)$ of a Krull domain $A$ is \[\cal{C}(A)=\cal{F}(\fk{X}(A))/\cal{H}(A)\]  where $\cal{F}(\fk{X}(A))$ is the free abelian group over the set of height-1 prime ideals and $\cal{H}(A)=\{a\bff{q}(A)\mid a\in \bff{q}(A)^\bullet\}$ is the subgroup of principal ideals of the quotient field. A Krull domain $A$ is factorial if and only if its class group $\cal{C}(A)$ is trivial. For more details, see \cite{FOSSUM,GH06}. Locally acyclic cluster algebras (\cite{GELS19}) and full rank upper cluster algebras (\cite{P24}) are Krull domains. It is immediate to see that this is also true for locally acyclic generalized cluster algebras (for the definition see \cite{BCDX20}), and for full rank generalized upper cluster algebras. The Markov cluster algebra is the only known example of a cluster algebra that is not a Krull domain (\cite[Section 6]{GELS19}). Nevertheless, the question of whether or not a (generalized) upper cluster algebra fails to be a Krull domain is still open. 
	
	\begin{theorem}[{\cite[Theorems 3.1 and 3.2]{GELS19}}]\label{thm:classgroupgeneralcase}
		Let $A$ be a Krull domain, and let $x_1,\dots,x_n\in A$ be such that $A[x_1^{-1},\dots,x_n^{-1}]=D[x_1^{\pm 1},\dots,x_n^{\pm 1}]$ is a factorial Laurent polynomial ring for some subring $D$ of $A$. Let $\mathfrak{p}_1,\dots,\mathfrak{p}_t$ be the pairwise distinct height-1 prime ideals of $A$ containing one of the elements $x_1,\dots,x_n$. Suppose that $$x_iA={\vprod_{l=1}^t} \mathfrak{p}_j^{a_{ij}},$$ with $\mathbf{a}_i=(a_{ij})_{j=1}^t\in \NN_0^t.$ Then $\mathcal{C}(A)\cong \ZZ^t/\langle\mathbf{a}_i\mid i\in[1,n]\rangle$ and $\cal{C}(A)$ is generated by $[\mathfrak{p}_1],\dots,[\mathfrak{p}_t]$.
		
		Suppose in addition that $D$ is infinite and either $n\ge 2$ or $n=1$ and $D$ has at least $|D|$ height-1 prime ideals. Then every class of $\mathcal{C}(A)$ contains precisely $|D|$ height-1 prime ideals.
	\end{theorem}
	
	As a direct application of the previous theorem we obtain the following result. 
	\begin{theorem}\label{thm:classgroup1}
		Let $(\bff{x},\rho, B)$ be a generalized seed, and $A$ be either the generalized cluster algebra $\cal{A}(\bff{x},\rho, B)$ or the generalized upper cluster algebra $\cal{U}(\bff{x},\rho, B)$. Assume that $A$ is a Krull domain. Let $\fk{p}_1,\ldots,\fk{p}_r$ be the pairwise distinct height-1 prime ideals of $A$ containing one of the elements $x_{1},\ldots,x_{n}$. Suppose that $$x_{i}A={\vprod_{l=1}^r} \mathfrak{p}_j^{a_{ij}},$$ with $\mathbf{a}_i=(a_{ij})_{j=1}^r\in \NN_0^r.$ Then $\mathcal{C}(A)\cong \ZZ^r/\langle\mathbf{a}_i\mid i\in[1,n]\rangle$ and $\cal{C}(A)$ is generated by $[\mathfrak{p}_1],\dots,[\mathfrak{p}_r]$.
		
		In addition, every class of $\cal{C}(A)$ contains precisely $|R|$ height-1 prime ideals.
	\end{theorem}
	\begin{proof}
		The theorem is a direct application of Theorem \ref{thm:classgroupgeneralcase}, since $A[\bff{x}^{-1}]=R[\bff{x}^{\pm1}]$, and $R[\bff{x}^{\pm1}]$ is a factorial Laurent polynomial ring. 
	\end{proof}
	
	\begin{remark} The algebraic structure of a Krull domain is completely determined (up to units) by its class group and the distribution of prime divisors in each class, see \cite[Theorem 2.5.4]{GH06}.
		Moreover, Theorem \ref{thm:classgroupgeneralcase} has important consequences in understanding factorization properties of a (generalized) cluster algebra. Let $\cal{A}$ be a (generalized) cluster algebra or a (generalized) upper cluster algebra, and $\cal{C}(\cal{A})$ its class group. If $\cal{C}(\cal{A})$ is trivial, then $\cal{A}$ is factorial. If the class group $\cal{C}(\cal{A})$ is infinite, then for every non-empty finite subset $L=\{l_1,\ldots,l_k\}$ of $\NN_{\ge 2}$ there exists $a\in \cal{A}$ such that \[a=u_{1,1}\cdots u_{1,l_1}=\cdots=u_{l_1,1}\cdots u_{l_1,l_1},\] where $u_{ij}$ are atoms of $\cal{A}$. This follows from a more general result of Kainrath on infinite class groups, see \cite{K99} or \cite[Theorem 7.4.1]{GH06}. If the class group $\cal{C}(\cal{A})$ is a finite group, then factorization properties of $\cal{A}$ can be studied via some arithmetic invariants, like the elasticity or the catenary degree. For more details, see the survey \cite{Sch16}. 
	\end{remark}
    \begin{examples} The previous theorems describe the structure of the class groups of certain (generalized) cluster algebras that are Krull domains; however, without explicitly identifying the height-1 prime ideals containing a cluster variable, an explicit computation of the class group is not possible. Here are two classes of examples where such computations have been carried out.
        \begin{enumerate}
            \item Let $\cal{A}=\cal{A}(\bff{x},B)$ be an acyclic cluster algebra of rank $n$ over the ring of integer $\ZZ,$ without isolated vertices. We say that two indices $i,j\in [1,n]$ are \defit{partner} if the corresponding exchange polynomials $f_i=x_ix_i', f_j=x_jx_j'$ do not share a non-trivial common factor. Partnership is an equivalence relation on $[1, n]$ and the equivalence classes are called partner sets. For a partner set $V\subseteq [1,n]$ and an positive integer $d\in \NN$, denote by $\mathsf{c}(V,d)$ the number of $i\in V$ for which $d$ divides $d_i:=\mathrm{gcd}\{b_{ji}\mid j\in [1,n]\}$ the greatest common divisor of the $i$-th column of $B$. Then the class group of $\cal{A}$ is isomorphic to $\ZZ^r$ with $$r=\sum_V\sum_{d\in \NN,\text{ odd}}(2^{\mathsf{c}(V,d)}-1)-|V|,$$ where the outer sum is taken over all partner sets $V\subseteq [1,n]$, see \cite[Theorem B]{GELS19}. A more general formula for cluster algebras over arbitrary fields can be found in \cite[Theorem 5.5, Remark 5.6]{GELS19}. See \cite[Section 5]{GELS19} for a long list of examples where the class group is explicitely computed, including class groups of algebras of finite type.
            \item Let $\mathcal{U}=\mathcal{U}(\bff{x},B)$ be a full rank upper cluster algebra of rank $n$ over a factorial domain $R$. For $i\in [1,n]$, let $l_i$ be the number of irreducible factors of the exchange polynomial $f_i=x_ix_i'$ in $R[\bff{x}].$ Then the class group of $\cal{U}$ is isomorphic to $\ZZ^r$, where $r=\sum_il_i-n$. For a proof and examples see \cite[Theorem 3.8]{P24}.
        \end{enumerate}
    \end{examples}
    \begin{remark}
        Recall that a domain is Krull if and only if it is completely integrally closed and satisfies the ACC on divisorial ideals. (Generalized) upper cluster algebras are intersections of factorial domains by construction, hence they are completely integrally closed. Therefore, if they fail to be Krull domains, they must not satisfy the ACC on divisorial ideals. 
    \end{remark}
	Classical (upper) cluster algebras have free class groups.
	
	\begin{theorem}\label{thm:classgroupclassical}
		Let $(\bff{x}, B)$ be a (classical) seed, and $A$ be either the cluster algebra $\cal{A}(\bff{x}, B)$ or the upper cluster algebra $\cal{U}(\bff{x}, B)$. Assume that $A$ is a Krull domain., and let $r\in \NN_0$ denote the number of height-1 prime ideals that contain one of the variables $x_{1},\ldots,x_{n}$. Then the class group $\cal{C}(A)$ of $A$ is free abelian of rank $r-n$. 
	\end{theorem}
    \begin{proof}
     For the proof see \cite[Theorem A]{GELS19} and \cite[Theorem 3.4]{P24} for the case of cluster algebras and upper cluster algebras, respecively. 
    \end{proof}
	\begin{remark}
		Theorem \ref{thm:classgroupclassical} and Theorem \ref{thm:classgroupgeneralcase} show us the different behavior of cluster algebras and generalized cluster algebras. The class group of a cluster algebra and  the class group of a generalized cluster algebra are always finitely generated and abelian. Exchange polynomials of a cluster algebra do not have repeated factors \cite[Proposition 2.3]{GELS19}, and this fact implies that the class group is torsion-free. In general, this is not the case for generalized cluster algebra, see Example \ref{ex:squarefactor} for an exchange polynomial with a square factor. Moreover, in Example \ref{ex:Z2} we show an example of a generalized cluster algebra with a torsion class group.  
	\end{remark}
	
	\begin{example}\label{ex:squarefactor}
		Consider the generalized seed $(\bff{x},\rho,B)$ given by 
		\begin{itemize}
			\item $\bff{x}=\{x_1,x_2\}$;
			\item $B=\begin{pmatrix}
			0 & -2 \\ 1 & 0	\end{pmatrix}$;
			\item $\rho_1=\{1,1\},$ $\rho_2=\{1,2,1\}$.
		\end{itemize}
		The exchange polynomials are \[f_1=x_2+1,\quad f_2=x_1^2+2x_1+1.\] Thus, $f_2=(x_1+1)^2$ has a square factor. 
	\end{example}
	
	\begin{proposition}\label{thm:height1primes}
		Let $(\bff{x},\rho, B)$ be a generalized seed, and $\cal{U}=\cal{U}(\bff{x},\rho, B)$ the generalized upper cluster algebra. Assume that $\cal{U}$ is a Krull domain. Then for every height-1 prime ideal $\mathfrak{p}$ of $\cal{U}$ there exists a cluster $\bff{y}$ such that $\mathfrak{p}L_{\bff{y}}\in \mathfrak{X}(L_{\bff{y}}).$
		
		Furthermore, if $\cal{U}=S_\bff{y}$ for some cluster $\bff{y}$, then for every height-1 prime ideal $\mathfrak{p}$ of $\cal{U}$ there exists $k\in [0,n]$ such that $\mathfrak{p}L_{\bff{y}_{k}}\in \mathfrak{X}(L_{\bff{y}_{k}}).$
	\end{proposition}
	\begin{proof}
		For the proof, one can repeat the same argument in \cite[Theorem 3.6]{P24}.
	\end{proof}
	\begin{corollary}\label{cor:boh}
		Let $(\bff{x},\rho, B)$ be a generalized seed and suppose that $\cal{U}(\bff{x},\rho, B)=S_\bff{x}$. Let $i\in [1,n]$, and consider the exchange polynomial $f_i=x_ix_i'$. Let $r_1,\ldots,r_t\in R[\bff{x}]$ be the distinct and non-associated irreducible polynomials that divide $f_i$. Then \[\{r_jL_{\bff{x}_i}\cap\cal{U}\mid j\in[1,t]\}\] is the set of all the height-1 prime ideals of $\cal{U}$ that contain $x_i$.
	\end{corollary}
	\begin{proof}
		Let $r_j$ be an irreducible element of $f_i$ in $R[\bff{x}]$. Since no monomial in $\bff{x}$ divides the exchange polynomials, $r_j$ is also irreducible, and hence prime, in $L_{\bff{x}_i}$. Denote by $\fk{p}'=r_jL_{\bff{x}_i}$ the height-1 prime ideal of $L_{\bff{x}_i}$ generated by $r_j$. Since $L_{\bff{x}_i}=\cal{U}[\bff{x}_i^{-1}]$, we have that $\fk{p}'\cap \cal{U}$ is a height-1 prime ideal of $\cal{U}$ that contains $x_i$.
		
		Vice versa, let $\fk{p}$ be a height-1 prime ideal of $\cal{U}$ that contains the cluster variable $x_i$. By Proposition \ref{thm:height1primes}, the ideal $\fk{p}L_{\bff{x}_k}$ is a height-1 prime ideal of $L_{\bff{x}_k}$ for some $k\in [1,n]$. Notice that $x_i$ is a unit in all the Laurent polynomials rings $L_{\bff{x}_j}$ with $j\ne i$, so the only possibility is that $\fk{p}L_{\bff{x}_i}$ is a height-1 prime ideal of $L_{\bff{x}_i}.$ In particular, $f_i=x_ix_i'\in \fk{p}L_{\bff{x}_i}$, hence there exists an irreducible factor $r_j$ of $f_i$ such that $r_j\in \fk{p}L_{\bff{x}_i}$. The element $r_j$ is irreducible also in $L_{\bff{x}_i}$, that is factorial, therefore $r_jL_{\bff{x}_i}=\fk{p}L_{\bff{x}_i}$, and $\fk{p}=r_jL_{\bff{x}_i}\cap \cal{U}$. 
		
	\end{proof}
	
	This theorem allows us to get more information about height-1 prime ideals that contain a cluster variable and hence more information on the class group. In the next example we show how to compute the class group of a generalized cluster algebra. 
	
	\begin{example}\label{ex:Z2}
		Consider the seed $(\bff{x},\rho, B)$ given in Example \ref{ex:squarefactor}. The directed graph $\Gamma(\bff{x},\rho, B)$ is \[\begin{tikzcd}
		1\ar[r] & 2
		\end{tikzcd}\] and it is acyclic. In addition, $\mathrm{det}(B)=2$, so $B$ has full rank and the seed is coprime. Theorem \ref{thm:acycliccoprime} then implies that \[\cal{A}(\bff{x},\rho, B)=\cal{U}(\bff{x},\rho, B)=S_\bff{x}.\]
		
		In particular, Corollary \ref{cor:boh} implies that $\fk{p}_1=(x_2+1)L_{\bff{x}_1}\cap \cal{U}$ is the only height-1 prime ideal that contains $x_1$ and that $\fk{p}_2=(x_1+1)L_{\bff{x}_2}\cap \cal{U}$ is the only height-1 prime ideal that contains $x_2$.
		Then by Corollary \ref{thm:classgroup1} we have $$\cal{C}(\cal{U})\cong \ZZ^2/\langle\bff{a}_1,\bff{a}_2\rangle,$$ where \[x_1\cal{U}=\fk{p}_1^{a_{11}}\cdot_v\fk{p}_2^{a_{12}},\quad \text{and} \quad x_2\cal{U}=\fk{p}_1^{a_{21}}\cdot_v\fk{p}_2^{a_{22}}.\] Hence, $$\cal{C}(\cal{U})\cong \ZZ^2/\langle(1,0),(0,2)\rangle\cong  \ZZ/2\ZZ.$$ 
	\end{example}
	\begin{theorem}
		Let $(\bff{x},\rho,B)$ be a generalized seed, and $\cal{U}=\cal{U}(\bff{x},\rho,B)$ the generalized upper cluster algebra associated to it. Assume that $\cal{U}=S_\bff{y}$ for some cluster $\bff{y}$. Then $\cal{U}$ is factorial if and only if the exchange polynomials $f_i=y_iy_i'\in R[\bff{y}]$ are irreducible.
	\end{theorem}
	
	\begin{proof}
		By Proposition \ref{prop:unit} we only need to prove that if the exchange polynomials $f_1,\ldots,f_n$ are irreducible, then $\cal{U}$ is factorial. By Theorem \ref{thm:classgroup1}, the class group $\cal{C}(\cal{U})$ of $\cal{U}$ is $\ZZ^r/\langle \bff{a}_1,\ldots,\bff{a}_n\rangle$, where  $\fk{p}_1,\ldots,\fk{p}_r$ are the height-1 prime ideals that contain one of the variables $y_1,\ldots,y_n$ and $\bff{a}_i=(a_{ij})_{j=1}^r$ is such that $$y_i\cal{U}=\fk{p}_1^{a_{i1}}\cdot_v\cdots \cdot_v \fk{p}_r^{a_{ir}}.$$ Corollary \ref{cor:boh} implies that $ \{f_iL_{\bff{x}_i}\cap \cal{A}\mid i\in [1,n]\}$ are the only height-1 prime ideals that contain $x_1,\ldots, x_n$, hence $r=n$ and $\fk{p}_i=f_iL_{\bff{x}_i}\cap \cal{U}$ for every $i\in [1,n]$, whence $a_{ij}=0$ for every $j\ne i$ and $a_{ii}=1$. Therefore $\langle \bff{a}_1,\ldots,\bff{a}_n\rangle\cong \ZZ^n$ and $\cal{C}(\cal{U})=0$, so $\cal{U}$ is factorial.
	\end{proof}
	\subsection{Every finitely generated abelian group is the class group of a generalized cluster algebra}
	\begin{center}
		\emph{Throughout all this section, the ground ring $R$ is an algebraically closed field.}
	\end{center}
	In this section, we show that every finitely generated abelian group can be realized as the class group of a generalized cluster algebra. We proceed showing some examples first.
    \begin{examples}
        \begin{enumerate}
            \item Let $G=\ZZ^n,$ with $n\ge 1.$ We want to construct a cluster algebra whose class group is isomorphic to $G$. Let $$B=\begin{pmatrix}
                0 & -1 \\ n+1 & 0
            \end{pmatrix},$$ and $\bff{x}=\{x_1,x_2\}.$ Let $\cal{A}$ be the cluster algebra associated to the acyclic coprime seed $(\bff{x},B)$ over an algebraically field $k$. Then $$f_1=x_2^{n+1}+1 \quad \text{ and } \quad f_2=x_1+1$$  are the exchange polynomials associated to $(\bff{x},B)$, and hence, by Corollary \ref{cor:boh}, the ideal $\fk{q}=f_2L_{\bff{x}_2}\cap \cal{A}$ is the only height-1 prime ideal that contains $x_2$, and $\fk{p}_j=r_jL_{\bff{x}_1}\cap \cal{A}$, where $r_1,\ldots,r_{n+1}$ are the irreducible factors of $f_1$ in $k[\bff{x}]$, are the unique height-1 prime ideals that contain $x_1$. Therefore, Theorem \ref{thm:classgroupclassical} implies that the class group of $\cal{A}$ is isomorphic to $G$.
            \item Let $G=\ZZ/n\ZZ.$  Let $$B=\begin{pmatrix}
                0 & -1 \\ n & 0
            \end{pmatrix},$$ $\bff{x}=\{x_1,x_2\},$ and $\rho_{1,r}=\binom{n}{r}$, $r\in [0,n]$, and $\rho_2=\{1,1\}$. Let $\cal{A}$ be the generalized cluster algebra associated to the acyclic coprime seed $(\bff{x},\rho,B)$ over any factorial domain $R$. Then $$f_1=\sum_{r=0}^n\binom{n}{r}x_2^r=(x_2+1)^n  \quad \text{ and }  \quad f_2=x_1+1$$ are the exchange polynomials associated to $(\bff{x},\rho,B)$. Again then, $\fk{q}=f_2L_{\bff{x}_2}\cap \cal{A}$ is the only height-1 prime ideal that contains $x_2$, and $\fk{p}=(x_2+1)L_{\bff{x}_2}\cap \cal{A}$ is the only height-1 prime ideal that contains $x_1$. Moreover, we have that \[x_1\cal{A}=\fk{p}^n \cdot_v \fk{q}^0, \quad \text{and} 
 \quad x_2\cal{A}=\fk{p}^0 \cdot_v \fk{q}^1 \]Therefore, Theorem \ref{thm:classgroup1} implies that the class group of $\cal{A}$ is isomorphic to $\ZZ^{2}/\langle (0,n),(1,0)\rangle$, that is isomorphic to  $\ZZ/n\ZZ.$
        \end{enumerate}
    \end{examples}
	\begin{theorem}\label{thm:realization}
		Let $G$ be a finitely generated abelian group. Then there exists an acyclic and coprime generalized cluster algebra $\cal{A}$ such that $\cal{A}$ is a Krull domain, its class group $\cal{C}(\cal{A})$ is isomorphic to $G$ and each class of $\cal{C}(\cal{A})$ contains exactly $|R|$ prime divisors.
	\end{theorem}
	\begin{proof} 	
		Let $m,k,n_1,\ldots,n_k\in \NN_0$ be such that $G\cong \ZZ^{m-1}\oplus\ZZ/n_1\ZZ\oplus\cdots\oplus\ZZ/n_k\ZZ$. Assume that $m\ge 1$, and that if $k>0$, then $0<n_1<\cdots<n_k$. Set $N=2k+2$, and suppose that $m+k>0$.

		Define a matrix $B=(b_{ij})$ of dimension $N\times N$ as follows:
		\[b_{ij}=\begin{cases}
		m & \text{if}\,\, i=N,j=1,\\
		n_j & \text{if}\, \,\frac{N}{2}< i\le N-1,j=N-i,\\
		-1 & \text{if}\,\, i\le \frac{N}{2},j=N-i+1,\\
		0 & \text{otherwise}.
		\end{cases}\]
		So $B$ is of the following form:     
  
  \vspace{1.5cm}
  
  \[B=\begin{pNiceArray}{c|ccc|cccc}
        0 & \Block{4-3}<\huge>{0}  & & & & & & -1\\
         \Vdots & & & & & & & \\
         & & &  & &\Iddots & &\\
         & & &  &-1 & & &\\
        \cline{2-8}
         & & & n_k& \Block{3-4}<\huge>{0} & & &\\
        & &  \Iddots &  & & & & \\
        & n_1  & & &  & & & \\
        m & 0 & \Cdots & & & & &
        \CodeAfter
 \OverBrace[shorten,yshift=3pt]{1-2}{1-4}{k}
 \OverBrace[shorten,yshift=3pt]{1-5}{1-8}{k+1}
    \end{pNiceArray}\]
		
		$B$ is a full rank skew-symmetrizable matrix. Set $d_i=n_i$, if $2\le i\le k+1$, and $d_i=1$ otherwise, and let $\bff{\rho}=(\rho_i)_{i\in[1,N]}$ be
		\[\rho_{i,r}=\begin{cases}
		{n_i \choose r} & \text{if}\,\, 2\le i\le k+1,\\
		1 & \text{otherwise},
		\end{cases}\qquad r\in[0,d_i].\]
		The seed $(\bff{x},\rho, B)$ is coprime and acyclic, hence Theorem \ref{thm:acycliccoprime} implies that \[A:=\cal{A}(\bff{x},\rho, B)=\cal{U}(\bff{x},\rho, B)=S_\bff{x}.\]
		The exchange polynomials associated to the initial seed are: 
		\[f_i=\begin{cases}
		x_{N}^m+1 & \text{if}\,\,i=1,\\
		\sum_{r=0}^{n_i}{n_i \choose r}x_{N-i+1}^r=(x_{N-i+1}+1)^{n_i} & \text{if}\,\, 2\le i\le k+1,\\
		x_{N-i+1}+1 & \text{if}\,\, i> k+1.
		\end{cases}\]
		Denote by $g_i$ the binomial $x_{N-i+1}+1$ for $i\in[2,N]$.
		Corollary \ref{cor:boh} implies that $\fk{p}_i=g_iL_{\bff{x}_i}\cap A$ is the only height-1 prime ideal of $A$ that contains the variable $x_i$ for every $i\in[2,N]$. Now factor  $f_1$ in irreducible factors, say $f_1=r_1\cdots r_m$, with $r_j\in R[\bff{x}]$ irreducible polynomials. Corollary \ref{cor:boh} implies again that $\{r_jL_{\bff{x}_1}\cap A\}$, $j\in [1,m]$, are the only height-1
		prime ideals of $A$ that contains $x_1$.
		Hence, \[\cal{C}(A)\cong \ZZ^{m+2k+1}/\langle\bff{a}_1,\ldots\bff{a}_{N}\rangle,\]with
		\[\bff{a}_1=(\underbrace{1,1,\ldots,1,}_{m}0,\ldots,0),\]and
		\[\bff{a}_i=\begin{cases}
		(\underbrace{0,\ldots,0,}_{m+i-1}n_i,0,\ldots,0) &\text{if}\,\,2\le i\le k+1,\\
		(\underbrace{0,\ldots,0,}_{m+i-1}1,0,\ldots,0) &\text{if}\,\,i> k+1.
		\end{cases}\]
		Consider now
		the map $\phi\colon\ZZ^{m+2k+1}\to G$ defined as follows: \[(u_1,\ldots,u_{m+2k+1}) \mapsto (u_1-u_2,u_2-u_3,\ldots,u_m-u_{m+1},u_{m+1}-u_1,[u_{m+2}]_{n_1},\ldots,[u_{m+1+k}]_{n_k}, 0,\ldots, 0),\]
		where $[u]_q$ is the $q$-class of $u$ in $\ZZ/q\ZZ$. It is easy to see that $\phi$ is a surjective homomorphism,
		and that  \begin{equation*}
		\begin{split}
		\ker\phi&=\{(u_{1},\ldots,u_{m+2k+1})\mid u_{m+1+i}\equiv 0 \Mod {n_i}, i\in[1,k], u_j=u_{j+1}, j\in [1,m+1]\}\\
		&=\{(\underbrace{\alpha,\ldots,\alpha,}_{m+1}n_1\beta_1,\ldots,n_k\beta_k,\gamma_1,\dots,\gamma_k)\mid \alpha,\beta_i,\gamma_i\in \ZZ\}\\
		&\cong \langle \bff{a}_1,\ldots,\bff{a}_{N}\rangle.
		\end{split}
		\end{equation*}
		Therefore $\cal{C}(A)\cong G$.
		
		Notice that each generalized cluster algebra that we constructed in the proof is acyclic and coprime, thus it is a Krull domain. The statement that each class of $\cal{C}(A)$ contains exactly $|R|$ prime divisors follows directly from Theorem \ref{thm:classgroup1}.
	\end{proof}
	
	\begin{example}We present an example to further illustrate the construction described in Theorem \ref{thm:realization}.
    
		Let $G=\ZZ\times \ZZ/3\ZZ$. Hence using the notation of Theorem \ref{thm:realization}, we have that $m=2, k=1, n_1=3,$ and in particular $N=4.$ So let $\bff{x}=\{x_1,x_2,x_3,x_4\}$ and construct the exchange matrix as follows:
        \[B=\begin{pNiceMatrix}
		0 & 0 & 0 & -1 \\
		0 & 0 & -1 & 0 \\
		0 & 3 & 0 & 0 \\
		2 & 0 & 0 & 0 \\
		\end{pNiceMatrix}.\] Set		
		 $\rho_1=\rho_3=\rho_4=\{1,1\}$, and $\rho_2,r=\binom{3}{r}$, so in particular $\rho_2=\{1,3,3,1\}.$ Then the exchange polymomials associated to $(\bff{x},\rho, B)$ are $$f_1=x_4^2+1, \quad f_2=x_3^3+3x_3^2+3x_3+1, \quad f_3=x_2+1, \quad f_4=x_1+1.$$ Let $\cal{A}=\cal{A}(\bff{x},\rho, B)$ be the cluster algebra over $\mathbb{C}$ associated to $(\bff{x},\rho, B)$. Then $\fk{p}_j=f_jL_{\bff{x}_j}\cap \cal{A}$, $j\in \{3,4\}$ are the only height-1 primes that contain $x_j$, $\fk{q}_1=(x_4+i)L_{\bff{x}_1}\cap \cal{A},$ and $\fk{q}_2=(x_4-i)L_{\bff{x}_1}\cap \cal{A}$ are the only height-1 primes that contain $x_1$, and $\fk{h}=(x_3+1)L_{\bff{x}_2}\cap \cal{A}$ is the only height-1 prime that contains $x_2$. Moreover, we have that \[x_1\cal{A}=\fk{q}_1\cdot_v \fk{q}_2, \quad x_2\cal{A}=\fk{h}^3, \quad x_3\cal{A}=\fk{p}_3, \quad x_4\cal{A}=\fk{p}_4,\] therefore the number of height-1 prime ideals that contain one cluster variable among $x_1,\ldots,x_4$ is $r=5$. Set \[\bff{a}_1=(1,1,0,0,0),\quad \bff{a}_2=(0,0,3,0,0),\quad \bff{a}_3=(0,0,0,1,0), \quad \bff{a}_4=(0,0,0,0,1).\] Then the class group $\cal{C}(\cal{A})$ is isomorphic by Theorem \ref{thm:classgroup1} to $\ZZ^r/\langle \bff{a}_1,\ldots,\bff{a}_4 \rangle$ that is isomorphic to $G$ via the map \[\phi\colon \ZZ^5 \to G, \quad (u_1,\ldots,u_5)\mapsto (u_1-u_2,u_2-u_1,[u_3]_3,0,0).\]
	\end{example}

	\section{LP algebras}\label{section:LP}
	This section is dedicated to Laurent phenomenon algebras (or LP algebras). They were introduced by Lam and Pylyavskyy in \cite{LP16} in order to get a general definition of those algebras for which the Laurent phenomenon holds. The aim of this section is to outline, step by step, that their behaviour is quite similar to the one of cluster algebras. A cluster algebra is an LP algebra under the condition of primitivity (the greatest common divisors of each column of any seed is 1) and coprimality (see Proposition \ref{prop:LP}). We show that the Markov cluster algebra is an LP algebra (see Example \ref{ex:markov}), even if the Markov seed is neither primitive nor coprime. In general an LP algebra and a cluster algebra coming from the same seed do not coincide (see Example \ref{ex:A3}).    
	\begin{definition}[LP seed]
		A \defit{Laurent phenomenon seed} (in short \defit{LP seed}) of rank $n$ in $\cal{F}$ is a pair $(\bff{x},\bff{F})$ such that
		\begin{enumerate}[label=(\roman*)]
			\item $\bff{x}=\{x_1,\ldots, x_n\}$ is a cluster, 
			\item $\bff{F}=\{F_1,\ldots,F_n\}$ is a collection of irreducible elements of $R[\bff{x}]$ such that for each $i,j\in [1,n]$, $x_j\nmid F_i$, and $F_i$ does not depend on $x_i$. The polynomials $F_1,\ldots, F_n$ are called the \defit{exchange polynomials} (associated to $(\bff{x},\bff{F})$). 
		\end{enumerate}
	\end{definition}
    \begin{example}
        Suppose that $R=\ZZ$ or $R=\mathbb{Q}$. Let $(\bff{x},B)$ be a classical seed of rank $n$, and set $d_i=\mathrm{gcd}\{b_{ji}\mid j\in [1,n]\}$. Let $\bff{F}=\{f_1,\ldots,f_n\}$ be the exchange polynomials associated to $(\bff{x},B).$ Then $(\bff{x},\bff{F})$ is a LP seed if and only if $d_i$ is a power of $2$ for every $i\in [1,n]$. If instead $R$ is an algebraically closed field, then $(\bff{x},\bff{F})$ is a LP seed if and only if $d_i=1$ for every $i\in [1,n].$  \cite[Proposition 2.3]{GELS19}
    \end{example}
	\begin{remark}
		By common convention, a polynomial $f\in R[\bff{x}]$ is an \defit{irreducible polynomial} if it is a non-constant polynomial such that cannot be factored into the product of two non-constant polynomials. In contrast, a polynomial $f\in R[\bff{x}]$ is an \defit{irreducible element} of $R[\bff{x}]$ if it is not a unit and it is not a product of two non-units. For instance, the polynomial $2\in \ZZ[x]$ is an irreducible element of $\ZZ[x]$, but it is not an irreducible polynomial of $\ZZ[x]$, and the polynomial $2x-4$ is an irreducible polynomial of $\ZZ[x]$ but it is not an irreducible element of $\ZZ[x]$. In the original definition of an LP algebra, Lam and Pylyavskyy \cite{LP16} assumed that the exchange polynomial $F_i$ is an {\em irreducible element of} $R[\bff{x}]$. This implies that the polynomials $F_i$ cannot be invertible. On the other hand, in the proof of \cite[Lemma 4.6]{DL22} it is implicitly assumed that if the exchange polynomial is constant, then it is invertible.
	\end{remark}
	Let $F,G \in R(x_{1},\ldots,x_{n})$ be two rational functions. Denote by $F|_{x_{i}\leftarrow G}$ the expression obtained by substituting $G$ for $x_{i}$ in $F$. If $x_i$ appears in $F$, we write $x_{i}\in F$. Otherwise, we write $x_{i}\notin F$.
	\begin{definition}[Exchange Laurent polynomials]
		Let $(\bff{x},\bff{F})$ be an LP seed. For each $F_j\in \bff{F}$, define a Laurent polynomial $\hat{F}_j=\frac{F_j}{x_1^{a_1}\cdots x_{j-1}^{a_{j-1}}x_{j+1}^{a_{j+1}}\cdots x_{n}^{a_{n}}}$, where $a_{k}\in \NN_{0}$ is maximal such that $F_{k}^{a_{k}}$ divides $F_{j}|_{x_{k}\leftarrow F_{k}/x}$ in  $R[x_{1},\ldots,x_{k-1},{x}^{-1},x_{k+1},\ldots, x_{n}]$. The Laurent polynomials $\bff{\hat{F}}=\{\hat{F}_{1},\ldots,\hat{F}_{n}\}$ are called the \defit{exchange Laurent polynomials} (associated to $(\bff{x},\bff{F})$).
	\end{definition}
	
    \begin{example}[{\cite[Example 2.5]{LP16}}]
        Let $R=\ZZ$, $\bff{x}=\{x_1,x_2,x_3\},$ and $\bff{F}=\{x_3+1,(x_1+1)^2+x_3^2,x_2^2+x_2+x_1^3+x_1^2\}.$ Then $\hat{F}_1=F_1$ and $\hat{F}_2=F_2$ since $x_1\in F_2,F_3$ and $x_2\in F_1,F_3.$ For the same reason, $x_2$ does not appear in the denominator of $\hat{F}_3$, since $x_3\in F_2$. So it remains to compute the exponent of $x_1$ in the denominator of $\hat{F}_3.$ In order to that, we replace in $F_3$ $x_1$ by $\frac{x_2+1}{x_1'}$, and we obtain \[g=x_2(x_2+1)+\frac{(x_2+1)^2(x_2+1+x_1')}{(x_1')^3}.\] Since $x_2+1$ divides $g$ in $\ZZ[(x_1')^{-1},x_2,x_3]$, but $(x_2+1)^m$ does not divide $g$ for all $m\in \NN_{\ge 2},$ we get that $\hat{F}_3=F_3/x_1.$
    \end{example}
	\begin{definition}[Mutation of LP seed]\label{def:LPmutation}
		Let $(\bff{x},\bff{F})$ be an LP seed, and $k\in [1,n]$. Define a new pair \[\mu_{k}(\bff{x},\bff{F}):=(\{x_{1},\ldots,x_{k}',\ldots,x_{n}\},\{F_{1}',\ldots,F_{k},\ldots,F_{n}'\}),\] where $x_{k}':=\hat{F}_{k}/x_{k}$ and the $F_{i}'s$ are obtained as follows:
		\begin{enumerate}
			\item If $x_{k}\notin F_{i}$, then $F_{i}'=F_{i}$;
			\item If $x_{k}\in F_{i}$, then $F_{i}' $ is obtained by the following steps:
			\begin{enumerate}
				\item\label{def:1} Define $G_{i}:=F_{i}|_{x_{k}\leftarrow N_{k}}$, where $N_{k}=\frac{\hat{F}_{k}|_{x_{i}\leftarrow 0}}{x_{k}'}.$
				\item Define $H_{i}$ to be $G_{i}$ with all common factors between $G_i$ and $\hat{F}_{k}|_{x_{i}\leftarrow 0}$ removed. Notice that $H_i$ is defined only up to units.
				\item Let $M$ be a Laurent monomial in $\{x_{1}',\ldots,x_{n}'\}$ such that $F_{i}'=MH_{i}\in R[\bff{x}']$ and it is not divisible by any variable in $\bff{x}'$. 
			\end{enumerate}
		\end{enumerate}
		Then we say that the new pair $\mu_{k}(\bff{x},\bff{F})$ is obtained from the LP seed $(\bff{x},\bff{F})$ by the \defit{LP mutation} in direction $k$.
	\end{definition}
	One can prove that $\mu_k(\bff{x},\bff{F})$ is also an LP seed, and that $\mu_k$ is an involution. See \cite{DL22} for more details. 
    \begin{examples}\label{ex:A3_0}
        Let $(\bff{x},\bff{F})$ be the LP seed giving by $$\bff{x}=\{x_1,x_2,x_3\},\qquad  \bff{F}=\{x_2+1,\,x_1+x_3,\,x_2+1\}.$$ By definition, we get that \[\hat{F}_1=\frac{F_1}{x_3},\qquad \hat{F}_2=F_2,\qquad \hat{F}_3=\frac{F_3}{x_1}.\] Let us mutate the seed in direction $1$, obtaining a seed $(\bff{x}',\bff{F}')$. Then by definition $\bff{x}'=\left\{\frac{x_2+1}{x_1x_3}, \, x_2,\,x_3\right\}$, and $F'_1=F_1.$ In order to compute $F_2'$ and $F_3'$ observe that $x_1\notin F_3,$ whence $F_3'=\epsilon F_3,$ with $\epsilon \in R^\times.$ On the other hand $x_1\in F_2,$ and so we need to proceed with the three steps above. In particular, $N_1=1/x_1'x_3$ and $G_2=1/x_1'x_3+x_3.$ Thus $H_2=\nu G_2,$ $\nu\in R^\times.$ Finally $F'_2=x_1'x_3^2+1.$
     \end{examples}
    \begin{remark}
        It is important to notice that, by definition, $F_i'$ is unique only up to units. And this is the motivation to consider LP seeds up to an equivalent relation.
    \end{remark} We say that $(\bff{y},\bff{G})$ is \defit{equivalent} to $(\bff{x},\bff{F})$ if for every $i\in[1,n]$ there exist $r_i,r_i'\in R^\times$ such that $x_i=r_iy_i$ and $F_i=r_i'G_i$. Denote by $[(\bff{x},F)]$ the set of LP seeds which are equivalent to $(\bff{x},F)$. 
    \begin{lemma}[{\cite[Lemma 3.1]{LP16}}]
        Let $(\bff{x},F),(\bff{y},G)$ be two LP seeds. If $[(\bff{x},F)]=[(\bff{y},G)],$ then $[\mu_k(\bff{x},F)]=[\mu_k(\bff{y},G)]$ for every $k\in [1,n].$
    \end{lemma} Let $(\bff{x},F)$ be an LP seed. By the above lemma, it is reasonable to define $\mu_k([(\bff{x},F)]):=[\mu_k(\bff{x},F)].$
	We say that $(\bff{y},\bff{G})$ is \defit{mutation-equivalent} to $(\bff{x},\bff{F})$ if there is a sequence of mutations $\mu_{k_m}\circ \cdots \circ \mu_{k_1}$ such that $\mu_{k_m}\circ \cdots\circ \mu_{k_1} ([(\bff{x},\bff{F})])=[(\bff{y},\bff{G})].$ Denote by $\mathcal{M}(\bff{x},\bff{F})$ the mutation equivalence class of $(\bff{x},\bff{F})$ and by $\mathcal{X}=\mathcal{X}(\bff{x},\bff{F})$ the set of all cluster variables appearing in $\mathcal{M}(\bff{x},\bff{F}).$
	
	\begin{definition}
		Let $(\bff{x},\bff{F})$ be an LP seed, and $\cal{X}=\cal{X}(\bff{x},\bff{F})$. 
  \begin{enumerate}
      \item The \defit{Laurent phenomenon algebra} (in short, \defit{LP algebra}) associated to $(\bff{x},\bff{F})$ is the $R$-algebra \[\cal{A}(\bff{x},\bff{F})=R[x\mid x\in \cal{X}].\]
      \item The \defit{upper LP algebra} associated to $(\bff{x},\bff{F})$ is the $R$-algebra  \[\cal{U}(\bff{x},\bff{F})=\bigcap_{(\bff{y},\bff{G})\in \cal{X}}{L_{\bff{y}}}.\] 
  \end{enumerate}The elements $x\in \mathcal{X}$ are called \defit{cluster variables}.
	\end{definition}
	The following theorem shows that LP algebras have the Laurent phenomenon.
	\begin{theorem}[{\cite[Theorem 5.1]{LP16}}]
		Let $(\bff{x},\bff{F})$ be an LP seed. Then \[\cal{A}(\bff{x},\bff{F})\subseteq \cal{U}(\bff{x},\bff{F}).\]
	\end{theorem}
	Recall that a matrix is \defit{primitive} if the great common divisor of each column is 1. 
	\begin{proposition}[{\cite[Proposition 4.4]{LP16}}]\label{prop:LP}
		Let $(\bff{x},B)$ be a classical seed. Assume that $B$ is full rank and primitive. Then the cluster algebra $\cal{A}(\bff{x},B)$ is an LP algebra, and for every seed of $\cal{A}(\bff{x},B)$, cluster algebra seed mutation agrees with LP algebra seed mutations.
	\end{proposition}
	
	\begin{example}\label{ex:A3}
		Let $(\bff{x},\bff{F})$ be the LP seed giving in Example \ref{ex:A3_0}. Denote by $\cal{A}_{LP}$ the LP algebra with initial seed $(\bff{x},\bff{F})$, and by $\cal{A}_C$ the cluster algebra with the same initial seed. It is well-known that the only cluster variables of $\cal{A}_C$ are \[\left\{x_1,\,x_2,\,x_3,\,\frac{1+x_2}{x_1},\,\frac{1+x_2}{x_3},\,\frac{x_1+x_3}{x_2},\,\frac{x_1+x_3+x_2x_3}{x_1x_2},\,\frac{x_1+x_3+x_1x_2}{x_2x_3},\frac{x_1+x_3+x_1x_2+x_2x_3}{x_1x_2x_3}\right\}.\]
		However, the cluster variables given by LP mutations are the following 
		\[\left\{x_1,x_2,x_3,\frac{1+x_2}{x_1x_3},\frac{x_1+x_3}{x_2},\frac{x_1+x_3+x_2x_3}{x_1x_2},\frac{x_1+x_3+x_1x_2}{x_2x_3}\right\}.\]
		As was shown in \cite[Example 1.12]{P24}, the element $\frac{1+x_2}{x_1x_3}\in \cal{A}_{LP}$ is not contained in the cluster algebra $\cal{A}_C$, hence the classical seed mutations and LP mutations give rise to two different algebras.

	\end{example}
	
	The following proposition shows some similarities between cluster algebras and LP algebras.
	
	\begin{proposition}\label{prop:prop}
		Let $(\bff{x},\bff{F})$ be an LP seed, and $\cal{A}=\cal{A}(\bff{x},\bff{F})$ \textup(resp., $\cal{U}=\cal{U}(\bff{x},\bff{F})$\textup) be the LP algebra (resp., upper LP algebra) associated to $(\bff{x},\bff{F})$. \begin{enumerate}[label=\textup(\normalfont{\arabic*\textup)}]
			\item $\cal{A}^\times=\cal{U}^\times=R^\times$.
			\item $\cal{A}$ and $\cal{U}$ are FF-domains.
			\item Every cluster variable is a strong atom of $\cal{A}$ and $\cal{U}$.
			\item The following statements are equivalent.
			\begin{enumerate}[label=\textup(\normalfont{\alph*}\textup)]
				\item $\cal{A}$ \textup(resp. $\cal{U}$\textup) is factorial;
				\item every cluster variable is prime in $\cal{A}$ \textup(resp., $\cal{U}$\textup);
				\item the cluster variables $x_1,\ldots,x_n$ are prime in $\cal{A}$ \textup(resp., $\cal{U}$\textup).
			\end{enumerate}
		\end{enumerate} 
	\end{proposition}
	\begin{proof}
		For (1), refer to \cite[Theorem 2.2]{GLS13}. The proof of (2) and (3) follows the same strategy as in Proposition \ref{prop:FF} and in Proposition \ref{prop:unit} respectively. For (4), refer to \cite[Corollary 1.23]{GELS19}
	\end{proof}
	
	\begin{example}[Markov Seed]\label{ex:markov}
		Let $(\bff{x},\bff{F})$ be the LP seed given by $\bff{F}=\{x_2^2+x_3^2,x_1^2+x_3^2,x_2^2+x_1^2\}$. By \cite[Lemma 2.7]{LP16}, $\bff{F}=\hat{\bff{F}}$. If $(\bff{y},\bff{G})$ is a seed mutation-equivalent to $(\bff{x},\bff{F})$, then $\bff{G}=\{y_2^2+y_3^2,y_1^2+y_3^2,y_2^2+y_1^2\}$ and hence $\hat{\bff{G}}=\bff{G}$. This implies that the (upper) LP algebra associated to $(\bff{x},\bff{F})$ is the Markov (upper) cluster algebra. In particular, $\cal{A}_{LP}(\bff{x},\bff{F})$ is not factorial. In addition, this example shows that the conditions in Proposition \ref{prop:LP} are not necessary conditions.
	\end{example}
	
	An LP seed $(\bff{x},\bff{F})$ is \defit{sign-skew symmetric} if $x_i\in F_j$ if and only if $x_j\in F_i$. Note that if a seed is sign-skew symmetric, then any mutations of it is sign-skew symmetric as well. 
	
	\begin{theorem}[{\cite[Corollary 4.14]{DL22}}]\label{thm:LPstarfish}
		Let $(\bff{x},\bff{F})$ be a sign-skew symmetric LP seed, and $\cal{U}=\cal{U}(\bff{x},\bff{F})$ be the upper LP algebra associated to it. If $G_k=\hat{G}_k$ for every $k\in [1,n]$ and for every $(\bff{y},\bff{G})\in \fk{X}(\bff{x},\bff{F})$, then $\cal{U}=\bigcap_{i=0}^nL_{\bff{x}_i}$.
	\end{theorem}
	\begin{remark}
		The proof of the previous theorem in \cite[Corollary 4.14]{DL22}, which relies on \cite[Lemma 4.6]{DL22}, implicitly assumes a skew-sign symmetric seed, even though this assumption is not explicitly mentioned by the authors.
	\end{remark}
	LP algebras that are Krull domains have a finitely generated free abelian class group.
	\begin{theorem}
		Let $(\bff{x}, \bff{F})$ be an LP seed, and $A$ be either the LP algebra $\cal{A}(\bff{x},\bff{F})$ or the upper LP algebra $\cal{U}(\bff{x}, \bff{F})$. Assume that $A$ is a Krull domain, and let $r\in \NN_0$ denote the number of height-1 prime ideals that contain one of the variables $x_{1},\ldots,x_{n}$. Then the class group $\cal{C}(A)$ of $A$ is free abelian of rank $r-n$. In particular, each class contains exactly $|R|$ height-1 prime ideals.
	\end{theorem}
	\begin{proof}
		The proof follows the same line of \cite[Theorem 3.4]{P24}.
	\end{proof}
	\begin{theorem}
		Let $\cal{U}$ be an upper LP algebra. If $\cal{U}=\bigcap_{k=0}^n L_{\bff{x}_k}$ for some cluster $\bff{x}$, then $\cal{U}$ is factorial.
	\end{theorem}
	\begin{proof}
		By Proposition \ref{prop:prop}, it is enough to show that $x_k$ is prime in $\cal{U}$ for every $k\in [1,n]$. We claim that it suffices to show that $x_k$ is prime in $L_{\bff{x}_k}$. Assume that $x_k$ is prime in $L_{\bff{x}_k}$, and suppose by contradiction that there exist $a,b\in \cal{U}$ such that $x_k\mid_{U} ab$, but $x_k$ does not divide $a$ and $b$ in $\cal{U}$. Since $x_k$ is a unit in $L_{\bff{x}_i}$ for every $i\in [1,n]\setminus\{k\}$, then $x_k$ does not divide $a,b$ in $L_{\bff{x}_k}$. Therefore we obtain a contradiction, because clearly $x_k\mid_{L_{\bff{x}_k}} ab.$	
		Now, the primeness of $x_k$ follows from the exchange relation and the irreducibility of $F_k$.
	\end{proof}
	
		\begin{remark}
		A necessary condition for a (upper) cluster algebra to be factorial is that exchange polynomials must be pairwise distinct and irreducible elements. From one hand exchange Laurent polynomials of an LP seed are irreducible by construction. Moreover, the authors did not find examples of an LP seed with two equal exchange Laurent polynomials. 
		\end{remark}

\section*{Acknowledgments}
The author would like to thank the anonymous referee for their careful reading and insightful comments. In particular, their suggestion to include more illustrative examples has greatly improved the clarity and presentation of the manuscript.

	\bibliographystyle{alpha}
	\bibliography{cluster.bib}
\end{document}